\documentclass[12pt]{amsart}

\usepackage[utf8]{inputenc}

\usepackage[english]{babel}
\usepackage{hyperref}
\usepackage[utf8]{inputenc}
\usepackage{graphicx}
\usepackage{graphics}
\usepackage{amssymb}
\usepackage{amsmath}
\usepackage{amsthm}
\usepackage{tikz-cd}
\usepackage{mathrsfs}
\usepackage[colorinlistoftodos]{todonotes}
\usepackage{enumitem}
\usepackage{yfonts}
\usepackage{lmodern}
\usepackage[T1]{fontenc}
\usepackage{import}
\usepackage{wrapfig}
\usepackage{calc}
\usepackage[all]{xy}
\usepackage{float}
\usepackage[new]{old-arrows}
\usepackage{mathtools}
\usepackage{subcaption}

\newtheorem{thm}{Theorem}[section]

\newtheorem{defn}[thm]{Definition}

\newtheorem{cor}[thm]{Corollary}

\newtheorem{prop}[thm]{Proposition}
\newtheorem{rmk}[thm]{Remark}

\newtheorem{lemma}[thm]{Lemma}
\def\R{\mathbb{R}}
\def\ep{\epsilon}
\def\C{\mathbb{C}}

\newtheorem*{theorem-non}{Theorem}

\def\ep{\epsilon}

\def\R{\mathbb{R}}
\def\Z{\mathbb{Z}}

\def\C{\mathbb{C}}

\def\Q{\mathbb Q}
\def\cal R{\mathcal R}

\title{Symplectic non-convexity of toric domains}

\author{Julien Dardennes}
\email{julien.dardennes@math.univ-toulouse.fr}
\address{Institute of Mathematics, University of Toulouse, 118 route de Narbonne, 31062 Toulouse Cedex 9, France}

\author{Jean Gutt}
\email{jean.gutt@math.univ-toulouse.fr}
\address{Institute of Mathematics, University of Toulouse,  118 route de Narbonne, 31062 Toulouse Cedex 9, France and INU Champollion, Place de Verdun, 81000 Albi, France}

\author{Jun Zhang}
\email{jun.zhang.3@umontreal.ca}
\address{Centre de Recherches Math\'ematiques, University of Montreal, C.P. 6128 Succ. Centre-Ville Montreal, QC H3C 3J7, Canada}

\begin{document}

\maketitle

\begin{abstract}
We investigate the convexity up to symplectomorphism (called symplectic convexity) of star-shaped toric domains in $\R^4$. In particular, based on the criterion from Chaidez-Edtmair in \cite{CD-20} via Ruelle invariant and systolic ratio of the boundary of star-shaped toric domains, we provide elementary operations on domains that can kill the symplectic convexity. These operations only result in $C^0$-small perturbations in terms of domains' volume. Moreover, one of the operations is a systematic way to produce examples of dynamically convex but not symplectically convex toric domains. Finally, we are able to provide concrete bounds for the constants that appear in Chaidez-Edtmair's criterion. 
\end{abstract}


\section{Introduction}
A toric domain $X_{\Omega}$ in $\R^4 (\simeq \C^2)$ is a symplectic manifold with boundary, defined by 
\[ X_{\Omega} := \mu^{-1}(\Omega) \]
where $\mu: \C^2 \to \R_{\geq0}^2$ is the moment map $(z_1, z_2) \to (\pi|z_1|^2, \pi|z_2|^2)$ and $\Omega \subset \R_{\geq 0}^2$ is a domain containing a neighborhood of the origin of $\R^2$. A toric domain $X_{\Omega}$ is called star-shaped if the radial vector field $\sum x_i \frac{\partial}{\partial x_i} + y_i \frac{\partial}{\partial y_i}$ of $\R^4$ is transverse to the boundary $\partial X_{\Omega}$ and intersects $\partial X_{\Omega}$ only once. This is equivalent to the condition, in the toric picture, that all the rays in $\R_{\geq 0}^2$ starting from the origin are transverse to $\partial \Omega$. In this paper, we always assume that $\partial \Omega$ is smooth. For a star-shaped domain $X_{\Omega}$, the boundary $\partial X_{\Omega}$ is a contact manifold with the contact structure given by the standard contact 1-form $\lambda_{\rm std} = \frac{1}{2} \left(\sum x_i dy_i - y_i dx_i\right)$.

\medskip

The study of the convexity of a symplectic toric domain $X_{\Omega}$ has a long history. As (geometric) convexity is not invariant under symplectomorphisms of $\R^4$, various substitutive notions, of symplectic nature, have been introduced. The most fruitful one is dynamical convexity, introduced in \cite{HWZ}, which requires the minimal Conley-Zehnder index of the Reeb orbits on $\partial X_{\Omega}$ has to be sufficiently positive. This serves as a dynamical hypothesis for many interesting problems. In this paper, we are mainly interested in the convexity of the toric domain up to symplectomorphism, for brevity, called {\it symplectic convexity}. A toric domain $X_{\Omega}$ is called {\it symplectically convex} if $\phi(X_{\Omega})$ is (geometrically) convex for some $\phi \in {\rm Symp}(\R^4)$. In \cite{HWZ}, a deep result says that any symplectically convex domain is dynamically convex, while, interestingly, only very recently has it been confirmed that 
\begin{equation} \label{Viterbo-question}
\mbox{dynamical convexity \,\,$\neq$ \,\, symplectic convexity on $S^3$\,\,\,\,by (\cite{CD-20})}.
\end{equation}
Verifying that a given toric domain $X_{\Omega}$ is symplectically convex is difficult since it requires to provide an explicit $\phi \in {\rm Symp}(\R^4)$ and to confirm that the image $\phi(X_{\Omega})$ is indeed convex. In this paper, we are making efforts towards a somewhat opposite direction: try to carry out a minimal amount of operations on a toric domain $X$ so that the resulting domain is {\it not} symplectically convex.

\medskip

To this end, a good characterization of a necessary condition for a toric domain to be symplectically convex is needed. Following Proposition 3.1 in \cite{CD-20}, this can be described as follows. Given a star-shaped domain $X_{\Omega}$, denote $Y: = \partial X_{\Omega}$. If $X_{\Omega}$ is symplectically convex, then
\begin{equation} \label{CD-criterion}
c \leq  {\rm ru}(Y, \lambda)\cdot {\rm sys}(Y,\lambda)^{\frac{1}{2}} \leq C
\end{equation}
where $c$ and $C$ are positive uniform constants (independent of the input $(Y, \lambda)$). The contact form $\lambda$ is the restriction of the standard contact form the boundary: $\lambda = \lambda_{\rm std}|_Y$. The quantities ${\rm ru}(Y,\lambda)$ and ${\rm sys}(Y,\lambda)$ are the {\it Ruelle ratio} and the {\it systolic ratio} of $(Y, \lambda)$ respectively. They are defined by 
\[ {\rm ru}(Y, \lambda) : = \frac{{\rm Ru}(Y, \lambda)}{{\rm Vol}(Y, \lambda)^{\frac{1}{2}}}, \,\,\,\,\mbox{and}\,\,\,\, {\rm sys}(Y, \lambda) = \frac{\mbox{(min period of a Reeb orbit)}^2}{{\rm Vol}(Y, \lambda)},\] 
where ${\rm Ru}(Y, \lambda)$ is the Ruelle invariant of the contact manifold $(Y,\lambda)$. For more details (and a complete definition), we refer to Section \ref{sec-ruelle}. 

It is important to emphasize that even though the ratios ${\rm ru}$ and ${\rm sys}$ are defined only on contact manifolds, here, the boundary of a star-shaped toric domain), they are in fact invariant under symplectomorphisms of $\R^4$ (see Proposition \ref{prop-invariant} in Section \ref{sec-cc} below). Moreover, the normalization by ${\rm Vol}(Y, \lambda)$ guarantees that both ratios ${\rm ru}$ and ${\rm sys}$ are invariant under the rescaling of the contact 1-form $\lambda$. We can, and we shall, therefore talk about the Ruelle ratio and the systolic ratio of a toric domain, that is, 
\[
	{\rm ru}\left(X_{\Omega}\right):={\rm ru}\left(\partial X_{\Omega}, \lambda\right)\,\,\,\,\, \mbox{and} \,\,\,\,\,{\rm sys}\left(X_{\Omega}\right):={\rm sys}\left(\partial X_{\Omega}, \lambda\right).
\]
The applications of the Ruelle invariant, originally defined in \cite{Rue-85}, in symplectic and contact geometry were initiated by \cite{Hutchings-Ruelle}. 

\medskip

The main result of this paper is the following:

\begin{thm} \label{thm-main-1} Given any star-shaped toric domain $X_{\Omega}$ in $\R^4$, there exist $C^0$-small perturbations of $X_{\Omega}$, in terms of the volume, such that the resulting domains $X_{\widehat{\Omega}}$ are still star-shaped but the product ${\rm ru}(X_{\widehat{\Omega}}) \cdot {\rm sys}(X_{\widehat{\Omega}})^{\frac{1}{2}}$ can be arbitrarily small or arbitrarily large. In particular, the resulting domains $X_{\widehat{\Omega}}$ are not symplectically convex. \end{thm}

The notation $X_{\widehat{\Omega}}$ indicates that the perturbations of $X_{\Omega}$ promised in Theorem \ref{thm-main-1} can be carried on directly on the toric image $\Omega$. Indeed, the proof of Theorem \ref{thm-main-1} in Section \ref{sec-proof-main} provides two explicit constructions of such perturbations. We apply (\ref{CD-criterion}) to deduce that $X_{\widehat{\Omega}}$ is not symplectically convex by deforming $\Omega$ until passing below the lower bound $c$ or over the upper bound $C$, even though we don't know explicitly how big $c$ and $C$ are in general. As expected, there will be non-trivial estimations of the ratios
\[ {\rm ru}\left(X_{\widehat{\Omega}}\right)\,\,\,\,\, \mbox{and} \,\,\,\,\,{\rm sys}\left(X_{\widehat{\Omega}}\right).\]
This is based on results from Section \ref{sec-ruelle} and beginning of Section \ref{sec-proof-main}. As a matter of fact, one construction of the perturbation yields (only) the first ratio to be arbitrarily large, while the other yields (only) the second ratio to be arbitrarily small. 

\begin{rmk} [Related to capacity] The operations in Theorem \ref{thm-main-1} do not always result in small perturbations in terms of symplectic capacities. It is possible to verify that $X_{\widehat{\Omega}}$ from method two - strain in subsection \ref{ssec-strain} - satisfies $X_{\Omega} \hookrightarrow X_{\widehat{\Omega}} \hookrightarrow hX_{\Omega}$ for a rescaling $h$ close to $1$, via symplectic folding technique (see \cite{Schlenk}), where $\hookrightarrow$ represents a symplectic embedding. However, the method one - strangulation in subsection \ref{ssec-strangulation} - is believed to change the symplectic capacities dramatically. How to describe such changes quantitatively in terms of capacities will be explored in further work. \end{rmk}

If one applies the deep result in \cite{HWZ}, symplectic convexity can be killed by breaking the dynamical convexity. In fact, this can be achieved also by a $C^0$-small perturbation in the sense of the volume. More explicitly, one simply  modifies the profile curve of $\Omega$ in $\R_{\geq 0}^2$ near the intersection points on the two axes so that the Conley-Zehnder indices of the Reeb orbits corresponding to these intersection points are always negative. We leave the details to interested readers. 

Here, we emphasize that our operations in Theorem \ref{thm-main-1} are fundamentally different. In particular, method one - strangulation in subsection \ref{ssec-strangulation} - can be distinguished with the operation elaborated above via symplectic capacities (for instance, the minimal action), while method two - strain in subsection \ref{ssec-strain} - can be carried out even within the category of dynamically convex domains. In particular, we have the following useful result. 

\begin{cor} \label{cor-Viterbo} For any dynamically convex toric domain $X_{\Omega}$ in $\R^4$, there exists a $C^0$-small perturbation in terms of the volume such that the resulting domain $X_{\widehat{\Omega}}$ is still dynamically convex but not symplectically convex. \end{cor} 

\begin{proof} This directly comes from the construction of strain operation in subsection \ref{ssec-strain}, Remark \ref{rmk-smt}, and \cite[Proposition 1.8]{HGR}. \end{proof}

Note that Corollary \ref{cor-Viterbo} provides a variety of examples that support (\ref{Viterbo-question}). In sharp contrast to the example produced in subsection 1.5 in (\cite{CD-20}) (which is closely related to the one invented in \cite{ABHS-invent}),  Corollary \ref{cor-Viterbo} above is to our best knowledge the first systematic way to produce {\it toric} examples to confirm (\ref{Viterbo-question}).

\medskip

As the proof of Theorem \ref{thm-main-1} is essentially fighting against the constant $c$ and $C$ appearing in (\ref{CD-criterion}), one may be curious about how small or large  these constants are. In general, due to the complexity of the proof of (\ref{CD-criterion}) in \cite{CD-20}, it seems difficult to read off any bounds for $c$ and $C$ directly. However, for a special family of star-shaped domains, called {\it monotone} toric domains (introduced in \cite{HGR}), we are able to estimate the product $ {\rm ru}(X_{\Omega})\cdot {\rm sys}(X_{\Omega})^{\frac{1}{2}}$ by the following result, which yields concrete bounds for $c$ and $C$ in (\ref{CD-criterion}) (cf \cite[Remark 1.11]{CD-20}). 

\begin{thm} \label{main-theorem-2}
For any monotone toric domain $X_{\Omega}$, we have 
\[ {\rm ru}(X_{\Omega})\cdot {\rm sys}(X_{\Omega})^{\frac{1}{2}} \geq \frac{1}{2}.\]
If, furthermore, $X_{\Omega}$ is (geometrically) convex in $\R^4$, then 
\[ {\rm ru}(X_{\Omega})\cdot {\rm sys}(X_{\Omega})^{\frac{1}{2}} \leq 3.\]
In particular, the constant $c$ and $C$ in the criterion (\ref{CD-criterion}) satisfies $c\leq \frac{1}{2}$ and $C \geq 3$, respectively. 
\end{thm}

Here, monotone means that the outward normal vectors $\nu=(\nu_1,\nu_2)$ at any point of the boundary component $\partial_+\Omega:=\partial \Omega \cap \R_{>0}^2$ have non-negative components ($\nu_1\geq0, \nu_2\geq0$). This is more general than the classical notion of convex or concave toric domains and admit many nice properties. For instance, by \cite[Proposition 1.8]{HGR}, the category of (strictly) monotone toric domains coincides with the category of dynamically convex domains. For another instance, due to \cite[Theorem 1.7]{HGR}, all normalized symplectic capacities agree on strictly monotone toric domains. 

\begin{rmk}
The proof of Theorem \ref{thm-main-1}, strain operation in subsection \ref{ssec-strain}, shows that within the category of strict monotone toric domains, there is no upper bound for the constant $C$ in (\ref{CD-criterion}), see Remark \ref{rmk-smt}. On the other hand, the lower bound $\frac{1}{2}$ in Theorem \ref{main-theorem-2} can be arbitrarily approximated. Indeed, consider the family of polydisks $P(a,b)$ with $b \to \infty$, we have 
\[ {\rm ru}(P(a,b)) \cdot {\rm sys}(P(a,b))^{\frac{1}{2}}  = \frac{(a+b)}{\sqrt{2ab}} \cdot \frac{a}{\sqrt{2ab}} = \frac{(a+b)}{2b} \to \frac{1}{2} \]
as $b \to \infty$. 
\end{rmk}

\begin{rmk} One may generalize the criterion (\ref{CD-criterion}) to characterize the geometric convexity of higher dimensional domains, where both the Ruelle invariant and systolic ratio are well-defined (see \cite{CD-higher}). Then, even if the constant $c$ and $C$ may change, our method in proving Theorem \ref{thm-main-1} can easily be adapted to the higher dimensional case.  
\end{rmk}

\noindent {\bf Acknowledgement} This work was completed while the third author held a CRM-ISM Postdoctoral Research Fellowship at the Centre de recherches math\'ematiques in Montr\'eal. He thanks this Institute for its warm hospitality. Part of Section 2 is derived from the project \cite{HZ-BM}, so he thanks Richard Hind for fruitful discussions. The second author was partially supported by a CIMI grant and the ANR COSY (ANR-21-CE40-0002) grant. We also thank Julian Chaidez for his comments on the initial version of the paper, in particular for the statement of Corollary \ref{cor-Viterbo}. 

\section{Ruelle invariant} \label{sec-ruelle}

For simplicity, let us assume our closed contact 3-manifold $(Y,\lambda)$ is a homology $3$-sphere. To the contact structure $\xi = \ker \lambda$, we can associate a real number $\text{Ru}(Y,\lambda)$ as follows. Consider the linearized Reeb flow at a point $y \in Y$ for time $t$ with respect to a trivialization $\tau$ which we will denote $\Phi^\tau_{y,t}$. For a real $T\geq 0$, the path $\Phi=\{\Phi^\tau_{y,t}\}_{t\in[0,T]}$ defines an element of the universal cover of the symplectic group $\widetilde{{\rm Sp}}(2)$. Together with the rotation number, $\rho:\widetilde{{\rm Sp}}(2)\rightarrow\mathbb{R}$ , this yields a real number $\rho(y,T,\tau):=\rho(\{\Phi^\tau_{y,t}\}_{t\in[0,T]})$ and the following limit
\begin{equation} \label{dfn-rot}
{\rm rot}(y)=\lim\limits_{T \rightarrow +\infty} \frac{\rho(y,T,\tau)}{T}
\end{equation}
is well-defined. In particular, ${\rm rot}(y)$ is independent of the trivialization $\tau$, since here, $Y$ has a unique trivialization up to homotopy. In general ${\rm rot}(y)$ only depends on the homotopy class of a trivialization. Moreover, by \cite{Rue-85}, we have that ${\rm rot}(y)$ is an integrable function. 

\begin{defn}\label{dfn-ruelle}
Suppose the closed contact 3-manifold $(Y,\lambda)$ is a homology $3$-sphere, then its Ruelle invariant is defined by 
$${\rm Ru}(Y,\lambda): =\int_Y {\rm rot}(y)\, \lambda\wedge d\lambda.$$
In particular, if $X$ is a star-shaped domain in $\mathbb{R}^4$, then we define $${\rm Ru}(X):={\rm Ru}(\partial X,\lambda_{\rm std}|_{\partial X})$$
\end{defn}

The following is the main result of this section.

\begin{prop} \label{prop-1} Let $X_{\Omega}$ be any 4-dimensional toric star-shaped domain. Then its Ruelle invariant is given by
\[ {\rm Ru}(X_{\Omega}) = a(\Omega) + b(\Omega) \]
where $a(\Omega)$ and $b(\Omega)$ are the $w_1$-intercept and $w_2$-intercept, respectively, of the moment image $\Omega$ in $\R_{\geq 0}^2$, in $(w_1, w_2)$-coordinate. \end{prop}

\begin{rmk} Proposition \ref{prop-1} generalizes  \cite[Proposition 1.11]{Hutchings-Ruelle} since there is no hypothesis that the profile curve, as the boundary $\partial \Omega \cap \R^2_{>0}$, has slopes everywhere negative (cf.~\cite[footnote on page 6]{Hutchings-Ruelle}). \end{rmk}

The proof of Proposition \ref{prop-1} occupies the rest of this section.

\subsection{Linearized Reeb flow} Denote by $\partial_+ \Omega := \partial \Omega \cap \R_{>0}^2$.  For any point $p = (w_1, w_2) \in \partial_+ \Omega$, consider the polar coordinate $(w_1, \theta_1, w_2, \theta_2)$. Then (recall that $\mu$ is the moment map) for any $z \in \mu^{-1}(p)$, one can verify that the Reeb vector field $R$ is
\begin{equation} \label{reeb-p}
R(z) = \frac{2\pi}{\nu_1(p) w_1 + \nu_2(p) w_2} \left(\nu_1(p)  \frac{\partial}{\partial{\theta_1}} + \nu_2(p)  \frac{\partial}{\partial{\theta_2}}\right)
\end{equation}
where $(\nu_1(p), \nu_2(p))$ is the unit normal vector of $\partial_+ \Omega$ at point $p$, pointing outward of $\Omega$. Note that $\nu_1(p) w_1 + \nu_2(p) w_2>0$ for any $p \in \partial_+ \Omega$ due to our hypothesis that $X_{\Omega}$ is star-shaped. Moreover, the contact 2-plane at $z$ is given by, 
\begin{equation} \label{xi-z}
\xi_z = \left\{ a_1 \frac{\partial}{\partial w_1} + b_1 \frac{\partial}{\partial \theta_1} + a_2 \frac{\partial}{\partial w_2} + b_2 \frac{\partial}{\partial \theta_2}\,\bigg| \, \begin{array}{c} \nu_1(p)a_1+\nu_2(p) a_2 = 0 \\ w_1 b_1 + w_2 b_2 = 0 \end{array} \right\},
\end{equation}
and one can choose a basis of $\xi_z$ as follows, 
\begin{equation} \label{basis}
e_1(p) = - \nu_2(p) \frac{\partial}{\partial w_1} + \nu_1(p) \frac{\partial}{\partial w_2}\,\,\,\,\,\mbox{and}\,\,\,\,\, e_2(p) = - w_2 \frac{\partial}{\partial \theta_1}  + w_1\frac{\partial}{\partial \theta_2}. 
\end{equation}
Note that $(e_1(p), e_2(p))$ is an {\it ordered} basis in that $(\omega_{\rm std})_z(e_1(p), e_2(p))>0$. Using this basis, along any Reeb trajectory $\gamma = (\gamma(t))_{t \in [0,T]}$, one can chose a trivialization $\tau: \gamma^*\xi \to \gamma\times \R^2$ explicitly defined as follows. For any $(z, v) \in (\gamma^*\xi)_z$ where $z \in \gamma$ and $v \in \xi_z$,  
\begin{equation} \label{trivialization-2} 
\tau(p)((z,v)) = (z, (v_R, v_{\theta})) \,\,\,\,\mbox{where $v = v_R e_1(p) + v_{\theta} e_2(p)$}. 
\end{equation}
Moreover, under this trivialization, the differentials of the Reeb flow along the trajectory $\gamma$ form a path in ${\rm Sp}(2)$, denoted by $\Phi$. The following lemma is crucial. 

\begin{lemma} \label{lemma-1} With respect to the trivialization given in (\ref{trivialization-2}), along the Reeb trajectory $\gamma = (\gamma(t))_{t \in [0,T]}$ the resulting path $\Phi$ in ${\rm Sp}(2)$ from the differentials of the Reeb flow is 
\[ \Phi = \left\{\begin{pmatrix} 1 & 0 \\ f(t) & 1 \end{pmatrix}\,\bigg| \, t \in [0, T] \right\}\]
where $f(t)$ is a linear function of $t$ depending only on $\gamma(0)$ and $e_1(\mu(\gamma(0))$ in (\ref{basis}). In particular, ${\rm rot}_{\tau}(\gamma(0), T) = 0$.  \end{lemma}

\begin{proof} Suppose $\gamma(0) \in \mu^{-1}(p)$ for some $p \in \partial_+ \Omega$. For $v \in \xi_{\gamma(0)}$ and any $t \in [0,T]$, to compute the differential $(d\phi_{R}^t|_{\xi_{\gamma(0)}})(v)$, we need to take a locally defined smooth path $r(s): (-\ep, \ep) \to \partial X_{\Omega}$ for $\ep$ sufficiently small such that $r(0) = \gamma(0)$ and $r'(0) = v$. Denote for brevity $r(s) = (w_1(s), \theta_1(s), w_2(s), \theta_2(s))$ where $w_1(0) =w_1$ and $w_2(0) = w_2$. For any $s \in (-\ep, \ep)$, by (\ref{basis}) and (\ref{trivialization-2}), 
\begin{align*}
r(s) &= r(0) + sv + o(s)\\
& = (w_1 - s\nu_2(p) v_R,\, \theta_1 - s w_2 v_{\theta},\, w_2 + s \nu_1(p) v_R, \,\theta_2 + s w_1 v_{\theta}) + o(s).
\end{align*}
Note that the approximation term $o(s)$ exist to guarantee that $r(s) \in \partial X_{\Omega}$. Then, by (\ref{reeb-p}), we have 
\begin{align*}
\phi_{R}^t(r(s)) & = \phi_{R_{\alpha}}^t(w_1(s), \theta_1(s), w_2(s), \theta_2(s)) \\
& = (w_1(s), \theta_1(s) + \Theta_1(s)\cdot t, w_2(s), \theta_2(s) + \Theta_2(s) \cdot t)
\end{align*}
where 
\[   \Theta_1(s) = \frac{2\pi \nu_1(p(s))}{\nu_1(p(s)) w_1(s) + \nu_2(p(s)) w_2(s)} \,\,\,\,\mbox{and}\,\,\,\, \Theta_2(s) = \frac{2\pi \nu_2(p(s))}{\nu_1(p(s)) w_1(s) + \nu_2(p(s)) w_2(s)},\]
where by notation $p(s) = (w_1(s), w_2(s))$. Observe that the denominator of $\Theta_1(s)$ and $\Theta_2(s)$ can be simplified as follows, 
\[ \nu_1(p(s)) w_1(s) + \nu_2(p(s)) w_2(s) = \nu_1(p(s)) w_1 + \nu_2(p(s)) w_2 + o(s). \]
In particular, it converges to $\nu_1(p) w_1 + \nu_2(p) w_2$ as $s \to 0$. Then, by the definition of a differential and computations above, 
\begin{align*}
(d\phi_{R}^t)(v) & = \lim_{s \to 0} \frac{\phi_{R}^t(r(s)) - \phi_{R}^t(r(0))}{s}\\
& = (-\nu_2(p) v_R, - w_2 v_{\theta}, \nu_1(p) v_R, w_1 v_{\theta}) \\
& + \left(0, \,\lim_{s \to 0} \frac{\Theta_1(s) - \frac{2\pi \nu_1(p)}{\nu_1(p) w_1 + \nu_2(p) w_2}}{s} \cdot t, \,0, \,\lim_{s\to0} \frac{\Theta_2(s) - \frac{2\pi \nu_2(p)}{\nu_1(p) w_1 + \nu_2(p) w_2}}{s} \cdot t\right).
\end{align*}
Meanwhile, further simplifications yield 
\begin{equation} \label{factor-1}
 \lim_{s \to 0} \frac{\Theta_1(s) - \frac{2\pi \nu_1(p)}{\nu_1(p) w_1 + \nu_2(p) w_2}}{s} = 2 \pi \cdot \frac{(\nu_1(p) \nu_2(p(s))'|_{s=0} - \nu_1(p(s))'|_{s=0} \nu_2(p))(-w_2)}{(\nu_1(p) w_1 + \nu_2(p) w_2)^2}, 
\end{equation}
and similarly, 
\begin{equation} \label{factor-2}
 \lim_{s \to 0} \frac{\Theta_2(s) - \frac{2\pi \nu_2(p)}{\nu_1(p) w_1 + \nu_2(p) w_2}}{s} = 2 \pi \cdot\frac{(\nu_1(p) \nu_2(p(s))'|_{s=0} - \nu_1(p(s))'|_{s=0} \nu_2(p))w_1}{(\nu_1(p) w_1 + \nu_2(p) w_2)^2}, 
 \end{equation}
where the $\nu_i(p(s))'|_{s=0}$ denotes the derivative with respect to the variable $s$ and then evaluated at $s=0$. For brevity, denote by 
\[ A(p; v): = 2 \pi \cdot \frac{\nu_1(p) \nu_2(p(s))'|_{s=0} - \nu_1(p(s))'|_{s=0} \nu_2(p)}{(\nu_1(p) w_1 + \nu_2(p) w_2)^2},\]
the common factor in (\ref{factor-1}) and (\ref{factor-2}). Then 
\begin{equation} \label{diff}
(d\phi_{R_{\alpha}}^t)(v) = (-\nu_2(p) v_R, (A(p; v) + v_{\theta})(-w_2)t, \nu_1(p) v_R, (A(p; v) + v_{\theta}) w_1 t). 
\end{equation}
In particular, 
\[ d\phi_{R}^t(e_1(p)) = e_1(p) + (A(p; e_1(p))t) e_2(p)\,\,\,\,\mbox{and}\,\,\,\, d\phi_{R}^t(v) = e_2(p). \]
Representing this by a matrix with respect to the basis $(e_1(p), e_2(p))$, one gets that
\begin{equation} \label{matrix}
d\phi_{R}^t|_{\xi_{\gamma(0)}} = \begin{pmatrix} 1 & 0 \\ A(p; e_1(p)) t & 1 \end{pmatrix}.
\end{equation}
Thus we prove the first conclusion by setting $f(t) := A(p; e_1(p)) t$. Moreover, the second conclusion is straightforward, since each matrix representation of the differential $d\phi_{R}^t|_{\xi_z}$ as in (\ref{matrix}) is similar to a shear matrix, which does not contribute any rotations. 
\end{proof}

\subsection{Proof of Proposition \ref{prop-1}}

Note that the trivialization in (\ref{trivialization-2}) does not extend to the entire $\partial X_{\Omega}$ (since the polar coordinate is not well-defined for the points where $w_1=0$ or $w_2=0$). However, for any globally defined trivialization $\bar{\tau}$, by Lemma \ref{lemma-1}, the rotation number at point $z \in \mu^{-1}(p)$ for $p \in \partial_+ \Omega$ is then only counting how much the function $\theta_1 + \theta_2$ changes along the Reeb flow, in other words, by (\ref{reeb-p}) and a similar argument as in \cite[top of page 10]{Hutchings-Ruelle}, we have 
\[ \rho(z) = \lim_{T \to \infty} \frac{{\rm rot}_{\bar{\tau}}(z, T)}{T} = \frac{(d\theta_1 + d\theta_2)(R_{\alpha}(z))}{2\pi} = \frac{\nu_1(p) + \nu_2(p)}{\nu_1(p) w_1 + \nu_2(p) w_2}\]
 where $p = (w_1, w_2)$. Note that $\rho(z)$ is in fact a function of $p \in \partial_+ \Omega$. Then by the definition of Ruelle invariant, 
\begin{align*}
{\rm Ru}(X_{\Omega}) & = \int_{\partial X_{\Omega}} \rho(z) \alpha \wedge d\alpha \\
& = \int_{\partial_+ \Omega} \frac{\nu_1(p) + \nu_2(p)}{\nu_1(p) w_1 + \nu_2(p) w_2} (w_1 dw_2 - w_2 dw_1). 
\end{align*}
Suppose the profile curve $\overline{\partial_+ \Omega}$ is parametrized by $\{(w_1(s), w_2(s))\}_{s \in [0,1]}$ such that 
\[ w_1(0) = a(\Omega), \,\, w_1(1) = 0 \,\,\,\,\mbox{and}\,\,\,\, w_2(0) = 0, \,\, w_2(1) = b(\Omega), \]
where $a(\Omega)$ and $b(\Omega)$ are the $w_1$-intercept and $w_2$-intercept. Moreover, by a change of variable, $w_1 dw_2 - w_2 dw_1 = (w_1(s) w'_2(s) - w_2(s) w'_1(s)) ds$. Meanwhile, observe that 
\begin{equation} \label{normal}
(\nu_1(p), \nu_2(p)) = (\nu_1(s), \nu_2(s)) = \left( \frac{-w'_2(s)}{\sqrt{|w'_1(s)|^2 + |w'_2(s)|^2}}, \frac{w'_1(s)}{\sqrt{|w'_1(s)|^2 + |w'_2(s)|^2}} \right).
\end{equation}
Therefore, by (\ref{normal}), 
\begin{align*}
{\rm Ru}(X_{\Omega}) &= \int_0^1 \frac{\nu_1(s) +\nu_2(s)}{\nu_1(s) w_1(s) + \nu_2(s) w_2(s)}(w_1(s) w'_2(s) - w_2(s) w'_1(s)) ds\\
& = \int_0^1 \frac{-w'_2(s) + w'_1(s)}{-w'_2(s) w_1(s) + w'_1(s) w_2(s)} (w_1(s) w'_2(s) - w_2(s) w'_1(s)) ds\\
& = \int_0^1 w'_2(s) - w'_1(s) ds\\
& = (w_2(1) - w_2(0)) -  (w_1(1) - w_1(0)) = b(\Omega) + a(\Omega). 
\end{align*}
Therefore, we obtain the desired conclusion.

\section{Convex criterion revisited} \label{sec-cc}

In this section, we prove that the convex criterion (\ref{CD-criterion}) is invariant under symplectomorphisms of $\R^4$. For brevity, in what follows in this section, $\lambda$ denotes the standard primitive $\lambda_{\rm std}$ of the standard symplectic structure $\omega_{\rm std}$ on $\R^4$. We have the following property. 

\begin{prop} \label{prop-invariant} Let $X$ and $X'$ be two star-shaped domains in $\R^4$. If $X$ and $X'$ are symplectomorphic, then ${\rm ru}(X) = {\rm ru}(X')$ and ${\rm sys}(X) = {\rm sys}(X')$. \end{prop}

Proposition \ref{prop-invariant} follows from a general statement, Lemma \ref{lemma-sc} below. It relates symplectomorphisms between star-shaped domains in $\R^4$ and strict contactomorphisms on their boundaries. Recall that a strict contactomorphism is defined as a contactomorphism that preserves the contact 1-forms.  

\begin{lemma}\label{lemma-sc}
Let $X$ and $X'$ be star-shaped domains of $\mathbb{R}^4$. Then, $X$ and $X'$ are symplectomorphic if and only if their contact boundaries $(\partial X,\lambda|_{\partial X})$ and $(\partial X',\lambda|_{\partial X'})$ are strictly contactomorphic.
\end{lemma}

\begin{proof} Assume that there is a symplectomorphism $\phi \in {\rm Symp}(\R^4)$ such that $\phi|_X(X) = X'$, in particular, $\phi|_{\partial X}(\partial X) = \partial X'$. Since $\phi^*\omega_{\rm std}=\omega_{\rm std}$ with $\omega_{\rm std}=d\lambda$, we have $d((\phi|_X)^*\lambda - \lambda) = 0$ everywhere on $X$. In particular, on the boundary $\partial X$, since $H_1(\partial X; \R) = H_1(S^3; \R)=0$, there exists a function $f$ on $\partial X$ such that $(\phi|_{\partial X})^*\lambda|_{\partial X'}=\lambda_{\partial X} +df$. Therefore, by \cite[Lemma 3.5]{CD-20} (which itself is due to Moser's trick), we know that the restriction $\phi|_{\rm \partial X}$ can be isotoped to a strict contactomorphism. 

On the other hand, assume that there is a strict contactomorphism $\psi$ such that $\psi(\partial X) = \partial X'$. Since $X$ is a Liouville domains with its completion being $\R^4$, we have the following structure decomposition 
\begin{equation} \label{symp-complete}
\R^4 \simeq \{\vec{0}\} \sqcup (\R \times \partial X) 
\end{equation}
where $\R \times \partial X$ is identified with the symplectization of $(\partial X, \lambda|_{\partial X})$ with respect to the symplectic structure $d(e^r \lambda|_{\partial X})$. We do the similar identification to $X'$. It is easily verified that a strict contactomorphism $\psi$ extends to a symplectomorphism $\Psi: (\R \times \partial X, d(e^r \lambda|_{\partial X})) \to (\R \times \partial X', d(e^r \lambda|_{\partial X'}))$. Indeed, the extension is explicitly given by 
\[ (r, x) \mapsto (r, \psi(x)). \]
In particular, the hypothesis that the contactomorphism $\psi$ is strict yields that the $\R$-factor $r$ is mapped to $r$. Therefore, $\Psi$ smoothly extends to the core part $\{\vec{0}\}$. Thus we obtained the desired symplectomorphism on $\R^4$, which maps $X$ onto $X'$ since both $X$ and $X'$ can be written as $\{\vec{0}\} \sqcup \{r<1\}$ under the identification (\ref{symp-complete}). \end{proof}

We are now ready to give the proof of Proposition \ref{prop-invariant}. 

\begin{proof} [Proof of Proposition \ref{prop-invariant}] By Lemma \ref{lemma-sc}, we are reduced to the hypothesis that $(\partial X,\lambda|_{\partial X})$ and $(\partial X',\lambda|_{\partial X'})$ are strictly contactomorphic by some strict contactomorphism $\psi$. Observe that ${\rm Vol}(\partial X, \lambda|_{\partial X}) = {\rm Vol}(\partial X', \lambda|_{\partial X'})$ since 
\begin{align*}
{\rm Vol}(\partial X', \lambda|_{\partial X'}) & = \int_{\partial X'} \lambda|_{\partial X'} \wedge (d \lambda|_{\partial X'}) \\
& = \int_{\psi(\partial X)}  \lambda|_{\partial X'} \wedge (d \lambda|_{\partial X'}) \\
& = \int_{\partial X}  \lambda|_{\partial X} \wedge (d \lambda|_{\partial X}) = {\rm Vol}(\partial X, \lambda|_{\partial X}). 
\end{align*}
Next, observe that $\phi$ maps the Reeb vector field $R$ on $\partial X$ to the Reeb vector field $R'$ on $\partial X'$ since 
\[ \lambda|_{\partial X'}(\phi_*(R)) = \lambda|_{\partial X}(R) = 1 \,\,\,\,\mbox{and}\,\,\,\, \iota_{\phi_*(R)} (d\lambda|_{\partial X'}) = \iota_{R} (d\lambda|_{\partial X}) = 0. \] 
This implies that the minimal period of the flow of $R$ is equal to the one of $R'$. Together with the above observation on the volumes, we obtain the second conclusion. 

Finally, since $\phi_*$ also preserves the linearized Reeb flow, by Definition \ref{dfn-ruelle}, we have ${\rm Ru}(\partial X, \lambda|_{\partial X}) = {\rm Ru}(\partial X', \lambda|_{\partial X'})$. Note that the trivialization is not involved in the discussion due to our assumption on the topology of $\partial X'$, even though the pushforward $\phi_*$ will potentially change the trivialization in the definition of the (limit of the) rotation number (\ref{dfn-rot}). Again, together with the conclusion on volume above, we obtain the first conclusion. \end{proof}

\begin{rmk} It is worth mentioning that for star-shaped domains, even though ${\rm Ru}(X)$, under our definition in Definition \ref{dfn-ruelle}, is invariant under the symplectomorphisms, it can not provide symplectic embedding obstructions in general. For instance, by Proposition \ref{prop-1}, the ellipsoid $E(1,4)$ symplectically embeds into the ball $B^4(2)$, while ${\rm Ru}(E(1,4)) = 5$ which is bigger than ${\rm Ru}(B^4(2)) =4$. In fact, Ruelle invariant behaves fundamentally different from symplectic capacities (cf.~\cite[Corollary 1.13]{Hutchings-Ruelle}). \end{rmk}

\section{Proof of Theorem \ref{thm-main-1}} \label{sec-proof-main}

In this section, we will give the proof of the main result, Theorem \ref{thm-main-1} in the introduction. Let us start from the following elementary observation on the closed Reeb orbit on a star-shaped domain $(\partial X, \lambda)$. By (\ref{reeb-p}), for a point $p = (w_1, w_2) \in \partial_+\Omega= \partial \Omega \cap \R_{>0}^2$, we know that
\begin{equation} \label{period-cond}
\mbox{a Reeb trajectory at $\mu^{-1}(p)$ is closed} \,\,\,\,\,\,\,\mbox{if and only if} \,\,\,\,\,\,\, \left|\frac{\nu_2(p)}{\nu_1(p)} \right|  \in \Q. 
\end{equation}
Assume that $\left|\frac{\nu_2(p)}{\nu_1(p)} \right|  \in \Q$. Denote by $h_p \in \R_{>0}$ the unique non-zero positive scalar such that 
\begin{itemize}
\item[(i)] $(h_p \nu_1(p), h_p \nu_2(p)) \in \Z^2$;
\item[(ii)] $h_p \nu_1(p), h_p \nu_2(p)$ are coprime. 
\end{itemize}
For brevity, denote $(m_p, n_p) := (h_p \nu_1(p), h_p \nu_2(p))$. The second condition (ii) above guarantees that the corresponding closed Reeb orbit, denoted by $\gamma_{(m_p, n_p)}$ is primitive (that is, not a multiple cover of another closed Reeb orbit). Note that the period of $\gamma_{(m_p, n_p)}$ is, by definition, the action $\mathcal A(\gamma_{(m_p, n_p)})$. Hence, 
\begin{equation} \label{action-1}
\mathcal A(\gamma_{(m_p, n_p)}) = (\nu_1(p) w_1 + \nu_2(p) w_2) \cdot h_p = m_p w_1 + n_p w_2.
\end{equation} 
This can be viewed as the inner product of the (integer-normalized) normal vector $(m_p, n_p)$ and the position vector $(w_1, w_2)$ (for point $p$). Due to our hypothesis that $X_{\Omega}$ is star-shaped, the action $\mathcal A(\gamma_{(m_p, n_p)})$ is always positive, even though the vector $(m_p, n_p)$ does not have both of its components positive.

\subsection{Method one: strangulation} \label{ssec-strangulation} Since $X_{\Omega}$ is a star-shaped domain, we can assume, for its moment image $\Omega$, without loss of generality, that the diagonal of $\R_{> 0}^2$ intersects $\partial \Omega$ at point $(w_*, w_*)$ such that a neighborhood within $\R_{>0}^2$ of the subset $\{(w, w)  \in \R_{>0}^2\, | \, 0\leq w < w_*\}$ lies in the interior of $\Omega$. In general, there always exists some ray in $\R_{\geq 0}^2$ satisfying this condition. By our assumption, since $\overline{\partial_+\Omega}$ is smooth, for any given $\ep>0$, there exists some angle $\theta(\ep)$ such that the unbounded sector with vertex $(\ep, \ep)$, divided in half by the diagonal, and angle equal to $2 \theta(\ep)$, intersects $\Omega$ in a closed region $\mathcal S(\ep)$ with points $(w_1, w_2) \in \partial \Omega \cap \mathcal S$ satisfying 
\[ |w_1 - w_*| \leq \ep \,\,\,\,\mbox{and}\,\,\,\, |w_2 - w_*| \leq \ep.\]
Now, carry on the following {\it strangulation} operation on $\Omega$, that is, define 
\[ \widehat{\Omega} : = \Omega \backslash ({\rm int}(\mathcal S(\ep)) \cup {\rm int}(\partial \Omega \cap \mathcal S(\ep))). \]
For a schematic picture of this operation, see Figure \ref{figure-1}. 
\begin{figure}[h]
\includegraphics[scale=0.8]{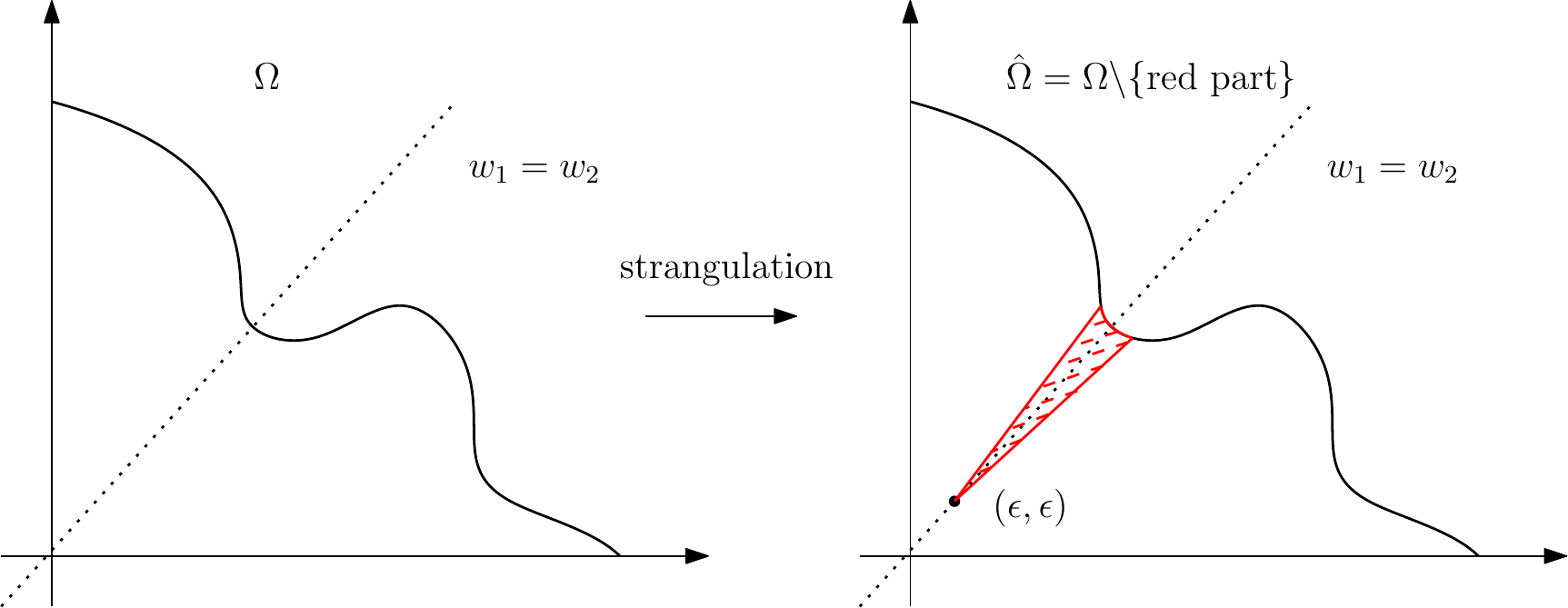}
\centering
\caption{Strangulation operation.} \label{figure-1}
\end{figure}

After a smoothening of all the possible corners in $\widehat{\Omega}$, we have that the resulting domain, still denoted by $\widehat{\Omega}$, is again a closed domain in $\R_{\geq 0}^2$ with its pre-image under the moment map $\mu^{-1}(\widehat{\Omega}) =: X_{\widehat{\Omega}}$ being a star-shaped domain. On the other hand, when $\ep<w_*$, the 4-dimensional volume of $X_{\Omega}$ and $X_{\widehat{\Omega}}$ (with respect to the standard symplectic structure on $\R^4$) satisfy 
\begin{equation} \label{vol}
|{\rm Vol}(X_{\widehat{\Omega}}) - {\rm Vol}(X_{\Omega})| \leq \pi \cdot 2(w_*+\ep)^2 \cdot \frac{\theta(\ep)}{\pi} = 8w_*^2 \cdot \theta(\ep)
\end{equation}
which goes to $0$ as $\ep$ goes to $0$ (since $\theta(\ep)$ goes to $0$). Therefore, $X_{\widehat{\Omega}}$ is indeed a $C^0$-small perturbation in terms of the volume of $X_{\Omega}$. 

By the discussion above on the closed Reeb orbits in (\ref{period-cond}), applied to the new domain $\widehat{\Omega}$, there exists a closed Reeb orbit at $p = (\ep, \ep)$ corresponding to the normal vector $(1,1)$. In particular, by (\ref{action-1}) its action is $1 \cdot \ep + 1 \cdot \ep = 2 \ep$. Denoting by $\widehat{T}_{\rm min}$ the minimal period of a closed Reeb orbit of $\partial X_{\widehat{\Omega}}$, we have 
\[ \widehat{T}_{\rm min} \leq 2 \ep \]
Therefore, 
\[ {\rm sys}\left(X_{\widehat{\Omega}}\right) \leq \frac{4 \ep^2}{{\rm Vol}(\partial X_{\widehat{\Omega}}, \lambda)} = \frac{4 \ep^2}{2{\rm Vol}(X_{\widehat{\Omega}})} \leq \frac{4 \ep^2}{2{\rm Vol}(X_{\Omega}) - 16w_*^2 \cdot \theta(\ep)} \]
where the equality comes from the following computation via Stokes' theorem, 
\begin{align*}
{\rm Vol}(\partial X_{\widehat{\Omega}}, \lambda)  = \int_{\partial X_{\widehat{\Omega}}} \lambda \wedge d\lambda &  = \int_{ X_{\widehat{\Omega}}} d(\lambda \wedge d\lambda)\\
& = \int_{X_{\widehat{\Omega}}} d\lambda \wedge d \lambda = 2{\rm Vol}(X_{\widehat{\Omega}}), 
\end{align*}
and the second inequality comes from (\ref{vol}). Hence, when $\ep$ is sufficiently small so that $w_*^2 \cdot \theta(\ep) < \frac{{\rm Vol}(X_{\Omega})}{16}$, we have 
\begin{equation} \label{sys-ep}
{\rm sys}\left(X_{\widehat{\Omega}}\right) \leq \frac{4 \ep^2}{{\rm Vol}(X_{\Omega})} \to 0 \,\,\,\,\mbox{as}\,\,\,\, \ep \to 0. 
\end{equation}
Finally, since the strangulation operation does not change the $w_1$ or $w_2$-intercepts of the original domain $\Omega$, due to Proposition \ref{prop-1}, the Ruelle invariant does not change, that is, ${\rm Ru}(X_{\widehat{\Omega}}) = {\rm Ru}(X_{\Omega})$. Therefore, similarly to the argument above, when $\ep$ is sufficiently small, we have 
\[ {\rm ru}\left(X_{\widehat{\Omega}}\right)^2 = \frac{{\rm Ru}(X_{\widehat{\Omega}})^2}{2{\rm Vol}(X_{\widehat{\Omega}})} \leq \frac{{\rm Ru}(X_{\Omega})^2}{2{\rm Vol}(X_{\Omega})-16w_*^2 \cdot \theta(\ep)} \leq2 \cdot {\rm ru}(X_{\Omega})^2, \]
where we used as before that $\ep$ is small enough so that $16w_*^2 \cdot \theta(\ep)< {\rm Vol}(X_{\Omega})$. Note that the right hand side is in particular finite; thus we have 
\[ {\rm ru}\left(X_{\widehat{\Omega}}\right) \cdot  {\rm sys}\left(X_{\widehat{\Omega}}\right)^{\frac{1}{2}}  \leq \frac{2 \ep}{\sqrt{{\rm Vol}(X_{\Omega})}}  \cdot \sqrt{2} \cdot {\rm ru}(X_{\Omega}). \]
Therefore, the product ${\rm ru}\left(X_{\widehat{\Omega}}\right) \cdot {\rm sys}\left(X_{\widehat{\Omega}}\right)^{\frac{1}{2}}$ will be lower than the constant $c$ appearing in criterion (\ref{CD-criterion}) whenever $\ep$ is sufficiently small by (\ref{sys-ep}). In conclusion, the domain $X_{\widehat{\Omega}}$ is not symplectically convex. 

\begin{rmk} It is not necessary to carry out the strangulation operation along the diagonal, as we did above. In general, most rays starting from the origin work in a similar way. An extreme case is to carry our such an operation along the $w_1$-axis or $w_2$-axis. The only difference is that the Ruelle invariant will change but gets smaller (so we still obtain the result that the product of ratios ${\rm ru} \cdot {\rm sys}$ will be eventually smaller than the constant $c$ in the criterion in (\ref{CD-criterion}). In fact, such an operation has been investigated in \cite{Usher-BM} on ellipsoids, called truncated ellipsoid. 

Theorem 1.4 in \cite{Usher-BM} shows that for any large number $A>0$, there always exists a truncated ellipsoid with its {\rm symplectic Banach-Mazur distance} $d_{\rm SBM}$ to the set of all ellipsoids larger than $A$. The distance $d_{\rm SBM}$ is a refinement of the symplectic version of the classical Banach-Mazur distance (see also \cite{SZ21}). It is a quantitative way to measure the symplectic embedding properties, but with some extra unknottedness condition (cf.~\cite{GU-knotted}). It is this unknottedness condition that prevents one from concluding that truncated ellipsoids are not symplectically convex, even though John's ellipsoid theorem readily implies that any symplectically convex domain should not be far from being squeezed by ellipsoids. 
\end{rmk}

\subsection{Method two: strain} \label{ssec-strain} Given a star-shaped domain $X_{\Omega}$, suppose that the $w_1$-intercept of $\overline{\partial_+\Omega}$ is $a >0$. Consider a generic $C^0$-small perturbation of $\Omega$ near $(a,0)$ but with the $w_1$-intercept $a$ fixed, which also results in a $C^0$-small perturbation of $X_{\Omega}$ in terms of the volume, such that in a neighborhood $N$ of $(a,0)$, the boundary $\overline{\partial_+\Omega}$ has a constant slope $k$, either positive or negative (but not equal to $\pm \infty$). This can be achieved due to our hypothesis that $\overline{\partial_+\Omega}$ is smooth, and we can consider $N$ sufficiently small so that the minimal period of the Reeb orbit of $\partial X_\Omega$ changes in an arbitrarily small way. For brevity, we still denote the domain after this perturbation by $\Omega$. 

Next, for any $\ep>0$, sufficiently small so that the (unique) point $(w_*(\ep), \ep) \in {\partial_+\Omega}$ for some $w_*>0$ lies in the neighborhood $N$ above, we have $\frac{\ep - 0}{w_*(\ep) - a} = k$, that is, 
\begin{equation} \label{w}
w_*(\ep) = \frac{\ep}{k} + a.
\end{equation}
Consider the following triangle 
\[ \mathcal T(\ep) := \mbox{the triangle determined by vertices $(0, 0)$, $(w_*(\ep), \ep)$, and $\left(\frac{1}{\sqrt{\ep}}, 0\right)$} \]
where $\ep$ is sufficiently small so that 
\begin{equation} \label{k}
\frac{-\ep}{\frac{1}{\sqrt{\ep}} - w_*(\ep)} > k \,\,\,\,\mbox{if $k<0$}
\end{equation}
This can be achieved since (\ref{k}) is equal to $k(a-\frac{1}{\sqrt{\epsilon}})>0$, so when $\ep \to 0$, we have $a - \frac{1}{\sqrt{\ep}}<0$ (since $k<0$). Then consider the following {\it strain} operation on $\Omega$, that is, 
\[ \widehat{\Omega} : = \Omega \cup \mathcal T(\ep). \]
For a schematic picture of this operation, see Figure \ref{figure-2}. 

\begin{figure}[h]
\includegraphics[scale=0.8]{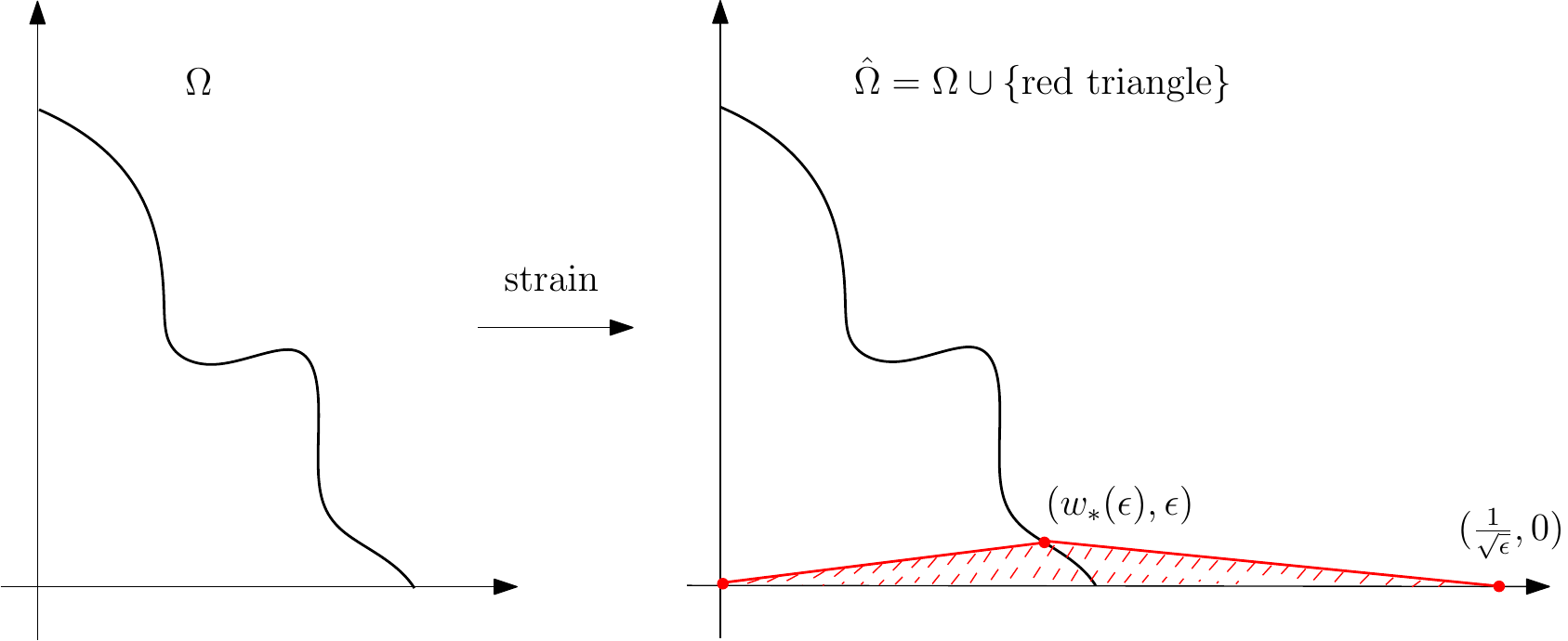}
\centering
\caption{Strain operation.} \label{figure-2}
\end{figure}

Observe that condition (\ref{k}) together with the hypothesis that $X_{\Omega}$ is star-shaped, implies that $\Omega \subset \widehat{\Omega}$ and the pre-image $X_{\widehat{\Omega}} = \mu^{-1}(\widehat{\Omega})$ is again star-shaped. In particular, $X_{\Omega}$ being star-shaped is used to deal with the case when $k >0$. 

Comparing the difference of the volume in $\R^4$, we have 
\begin{equation} \label{vol-2}
|{\rm Vol}(X_{\widehat{\Omega}}) - {\rm Vol}(X_{\Omega})| \leq \frac{\ep \cdot \frac{1}{\sqrt{\ep}}}{2} = \frac{\sqrt{\ep}}{2} 
\end{equation}
which goes to $0$ as $\ep$ goes to $0$. Therefore, $X_{\widehat{\Omega}}$ is indeed a $C^0$-small perturbation in terms of the volume of $X_{\Omega}$. 

This operation possibly introduces various new closed Reeb orbits. Besides the one corresponding to the $w_1$-intercept point $(\frac{1}{\sqrt{\ep}}, 0)$ with large action, others will concentrate only near the point $p = (w_*(\ep), \ep)$, after a smoothening at $p$. By (\ref{period-cond}), these closed Reeb orbits correspond to the pairs of integers, 
\[ (m_p, n_p) \in \Z_{>0} \times \Z \,\,\,\,\mbox{with} \,\,\,\, \min\left\{- \frac{1}{k}, 0\right\} \leq \frac{n_p}{m_p} \leq \frac{\sqrt{\ep} - w_*(\ep)}{\ep}. \]
Concerning their action, we have by (\ref{action-1}),  
\begin{align*}
\mathcal A(\gamma_{(m_p, n_p)}) & = m_p w_*(\ep) + n_p \ep \\
& = m_p\left(w_*(\ep) + \frac{n_p}{m_p} \ep \right)\\
& \geq m_p \left( \frac{\ep}{k} + a + \min\left\{- \frac{\ep}{k}, 0\right\}\right) \geq \frac{a}{2},
\end{align*}
when $\ep$ is sufficiently small. We denote  as above, $T_{\rm min}$ the minimal period of a closed Reeb orbit on $\partial X_{\Omega}$ and $\widehat{T}_{\rm min}$ the minimal period of a closed Reeb orbit on $\partial X_{\widehat{\Omega}}$. If $T_{\rm min} < \frac{a}{2}$, then obviously $\widehat{T}_{\rm min} = T_{\rm min}$. If, on the other hand, $T_{\rm min} \geq \frac{a}{2}$, then $\widehat{T}_{\rm min} \geq \frac{a}{2}$. Meanwhile, by assumption, $T_{\rm min} \leq a$. Therefore, in either case, we have
\[ \widehat{T}_{\rm min} \geq \frac{T_{\rm min}}{2}. \]
In particular, the lower bound $\frac{T_{\rm min}}{2}$ is independent of the parameter $\ep$. 

Now, for the ratios in discussion, by (\ref{vol-2}), 
\[ {\rm sys}\left(X_{\widehat{\Omega}}\right) \geq \frac{T_{\rm min}^2}{4{\rm Vol}(\partial X_{\widehat{\Omega}}, \lambda)} = \frac{T_{\rm min}^2}{8{\rm Vol}(X_{\widehat{\Omega}})} \geq \frac{T_{\rm min}^2}{8{\rm Vol}(X_{\Omega}) + 4\sqrt{\ep}} \]
When $\ep$ is sufficiently small, say $\sqrt{\ep}< {\rm Vol}(X_{\Omega})$, we have 
\begin{equation} \label{sys-lb}
{\rm sys}\left(X_{\widehat{\Omega}}\right) \geq \frac{1}{12} \cdot \frac{T_{\rm min}^2}{{\rm Vol}(X_{\Omega})}  \,\,\,\,\mbox{which is independent of $\ep$}. 
\end{equation}
Finally, the strain operation results in an essential change of the Ruelle invariant. By Proposition \ref{prop-1}, 
\[ {\rm Ru}(X_{\widehat{\Omega}}) = \mbox{($w_2$-intercept of $\overline{\partial_+\Omega}$)} + \frac{1}{\sqrt{\ep}} \geq \frac{1}{\sqrt{\ep}}. \]
Similarly as above, 
\[ {\rm ru}\left(X_{\widehat{\Omega}}\right)^2 = \frac{{\rm Ru}(X_{\widehat{\Omega}})^2}{2{\rm Vol}(X_{\widehat{\Omega}})}   \geq \frac{1}{3\ep} \cdot \frac{1}{{\rm Vol}(X_{\Omega})}.\]
Then we have 
\[ {\rm ru}\left(X_{\widehat{\Omega}}\right) \cdot {\rm sys}\left(X_{\widehat{\Omega}}\right)^{\frac{1}{2}} \geq \frac{1}{6 \ep} \cdot \frac{T_{\rm min}}{{\rm Vol}(X_{\Omega})} \to +\infty \,\,\,\,\mbox{as $\ep \to 0$}.\]
Hence, the product of the ratios will be larger than the constant $C$ appearing in criterion (\ref{CD-criterion}). Therefore, the domain $X_{\widehat{\Omega}}$ is not symplectically convex. 

\begin{rmk} \label{rmk-smt} Note that the strain operation is closed within the category of strictly monotone toric domains (hence, within the category of dynamically convex domains by \cite[Proposition 1.8]{HGR}), since by the definition of a monotone toric domain, near the $w_1$-intercept the corresponding slope $k$ is always negative.  Then the strain method above implies that the product of ratios ${\rm ru} \cdot {\rm sys}^{\frac{1}{2}}$ is unbounded for monotone toric domains.\end{rmk}

\section{Estimate of constants} \label{sec-estimation}

Recall that a strictly monotone toric domain is a star-shaped domain such that the outward normal vectors along the boundary component $\partial_+\Omega=\partial \Omega \cap \R_{>0}^2$ all have both components positive. In contrast with Remark \ref{rmk-smt}, in this section, we give a proof of Theorem \ref{main-theorem-2}, which provides a uniform bounds of the product of ratios ${\rm ru} \cdot {\rm sys}^{\frac{1}{2}}$ for such domains (when they are geometrically convex in $\R^4$). 

\medskip 

Let us start from the following useful result. Denote by $c_{\rm Gr}(X_{\Omega})$ the Gromov width of a toric domain $X_{\Omega}$, measuring the largest $B^4(a)$ that can be symplectically embedded into $X_{\Omega}$. 

\begin{lemma} \label{lem-vol-gr}
Let $X_\Omega$ be a monotone toric domain where the $w_1$-intercept and $w_2$-intercept of $\overline{\partial_+\Omega}$ are $(a,0)$ and $(0,b)$ respectively. Suppose that $b\geq a$, then
\[ {\rm Vol}(X_\Omega)\leq b\cdot c_{\rm Gr}(X_\Omega).\]
\end{lemma}

\begin{proof} By the proof of \cite[Theorem 1.11]{HGR}, we know that $c_{\rm Gr}(X_{\Omega})$ is equal to the largest $L>0$ such that the line $w_2 = - w_1 +L$ touches the boundary $\overline{\partial_+\Omega}$ for the first time. Denote by $(s, t)$ one of these intersection points. Then observe that $X_{\Omega}$ being monotone implies that $X_{\Omega} \subset P(s, b) \cup P(a, t).$ Therefore, we have
\begin{align*}
{\rm Vol}(X_{\Omega})  \leq sb + at & =  sb + a(-s + c_{\rm Gr}(X_{\Omega}))\\
& = s(b-a) + a c_{\rm Gr}(X_{\Omega}) \leq b \cdot c_{\rm Gr}(X_{\Omega}) 
\end{align*}
where the last inequality comes from $s \leq c_{\rm Gr}(X_{\Omega})$. \end{proof}

Now, we are ready to give the proof of Theorem \ref{main-theorem-2}. 

\begin{proof} [Proof of Theorem \ref{main-theorem-2}] For a monotone toric domain $X_{\Omega}$, \cite[Theorem 1.7]{HGR} shows that all normalized symplectic capacities coincide. In particular, the minimal period of a Reeb orbit is equal to $c_{\rm Gr}(X_{\Omega})$. Without loss of generality, assume $b \geq a$. Then by Proposition \ref{prop-1} and Lemma \ref{lem-vol-gr}, we have 
\begin{align*}
   \text{ru}(X_\Omega)^2 \cdot \text{sys}(X_\Omega)&  =\frac{(a+b)^2}{4} \cdot \frac{c_{\rm Gr}(X_\Omega)^2}{\text{Vol}(X_\Omega)^2}\\
    & \geq \frac{(a+b)^2}{4}\frac{c_{\rm Gr}(X_\Omega)^2}{(b\cdot c_{\rm Gr}(X_\Omega))^2}\\
    & \geq \frac{(a+b)^2}{4b^2}=\frac{1}{4}\left(1+\frac{a}{b}\right)^2  \geq \frac{1}{4}.
\end{align*}
Thus, we complete the proof of the first conclusion. 

\medskip

Now, suppose furthermore that $X_\Omega$ is geometrically convex in $\R^4$. Up to a rescaling, assume the $w_1$-intercept of $\Omega$ is $1$ while the $w_2$-intercept of $\Omega$ is still $b$. Up to a reflection between $w_1$ and $w_2$, we can assume that $b \geq 1$.
Therefore, we have $\text{Ru}(X_\Omega)=1+b$ for any such domain and so
\[ \text{ru}(X_\Omega)^2 \cdot \text{sys}(X_\Omega)=\frac{(1+b)^2c_{\rm Gr}(X_\Omega)^2}{4\text{Vol}(X_\Omega)^2}\]
depends only on $\frac{c_{\rm Gr}(X_\Omega)}{\text{Vol}(X_\Omega)}$. We thus aim to bound above this quantity among monotone toric domains which are geometrically convex.

By \cite[Proposition 2.3]{HGR}, the following subset 
\[ \widetilde{\Omega} : = \{(\mu_1, \mu_2) \in \R^2 \,| \, (\pi |\mu_1|^2, \pi|\mu_2|^2) \in \Omega \}\]
is a convex subset in $\R^2$. In particular, when restricted to $\R_{\geq 0}^2$, the boundary $\partial \widetilde{\Omega}$ can be written as a decreasing concave function $\mu_2 = g(\mu_1)$. Since $g$ is concave, we have $g(\mu_1) \geq \sqrt{b}(1-\mu_1)$ for all $\mu_1 \in [0,1]$. Meanwhile, if we fix $c=c_{\rm Gr}(X_\Omega)$, we also have $g(\mu_1)\geq \sqrt{c-\mu_1^2}$ for all $\mu_1 \in [0,1]$ (since $B^4(c)\subset X_\Omega$). Therefore, $g$ is above the broken curve consisting of the two previous curves. Hence, among these $g$, the one whose domain maximizes $\frac{c_{\rm Gr}}{\text{Vol}}$ is the one minimizing the volume i.e. the convex hull of this broken curve, see Figure \ref{figure-3}. It has the following boundary:
$$g_c(\mu_1)=\begin{cases}
\sqrt{b}-\sqrt{\frac{b-c}{c}}\mu_1 &\text{ if }\,\,\,\,0\leq \mu_1\leq \sqrt{\frac{c}{b}(b-c)}\\
\sqrt{c-\mu_1^2} &\text{ if } \,\,\,\,\sqrt{\frac{c}{b}(b-c)}\leq \mu_1\leq c\\
\sqrt{\frac{c}{1-c}}(1-\mu_1) &\text{ if }\,\,\,\, c\leq \mu_1\leq 1
\end{cases}.$$
\begin{figure}[h]
\includegraphics[scale=0.8]{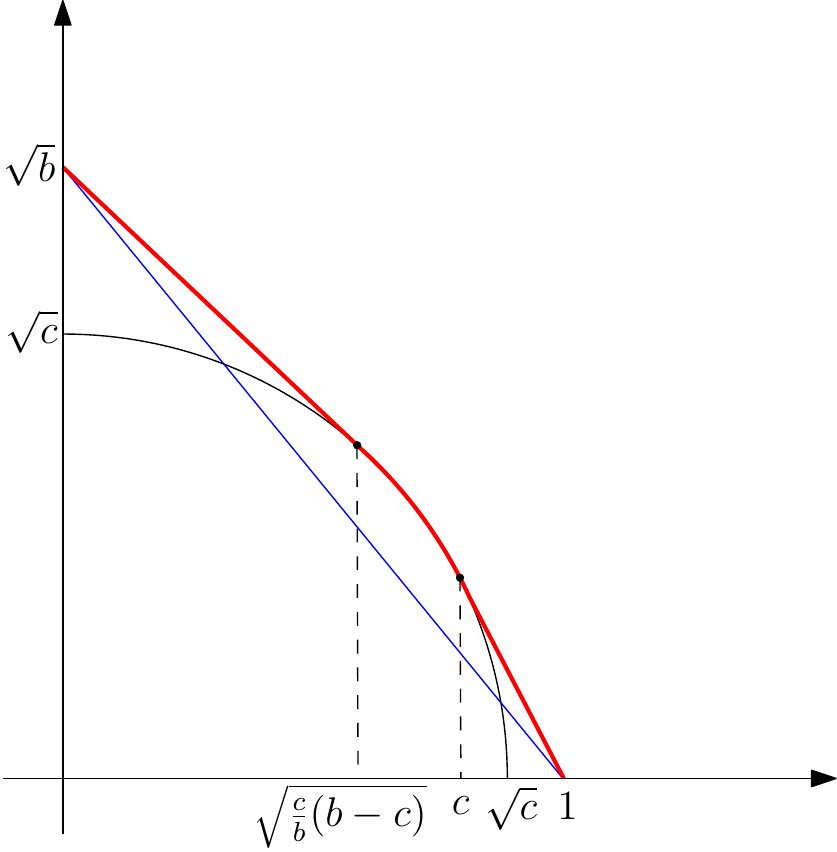}
\centering
\caption{Graph of $g_c$ in red.} \label{figure-3}
\end{figure}
 By a change of variables $w_i=\mu_i^2$, we know that the boundary $\overline{\partial_+\Omega}$ (minus the components on $w_1$-axis and $w_2$-axis) is a function $w_2 = f_c(w_1): = g_c(\sqrt{w_1})^2$ given by

$$f_c(w_1)=\begin{cases}
\left(\sqrt{b}-\sqrt{\frac{b-c}{c}w_1}\right)^2 &\text{ if }\,\,\,\,0\leq w_1\leq \frac{c}{b}(b-c)\\
c-w_1 &\text{ if }\,\,\,\, \frac{c}{b}(b-c)\leq w_1\leq c^2\\
\frac{c}{1-c}(1-\sqrt{w_1})^2 &\text{ if } \,\,\,\,c^2\leq w_1\leq 1
\end{cases},$$
also see Figure \ref{figure-4}. 
\begin{figure}[h]
\includegraphics[scale=0.8]{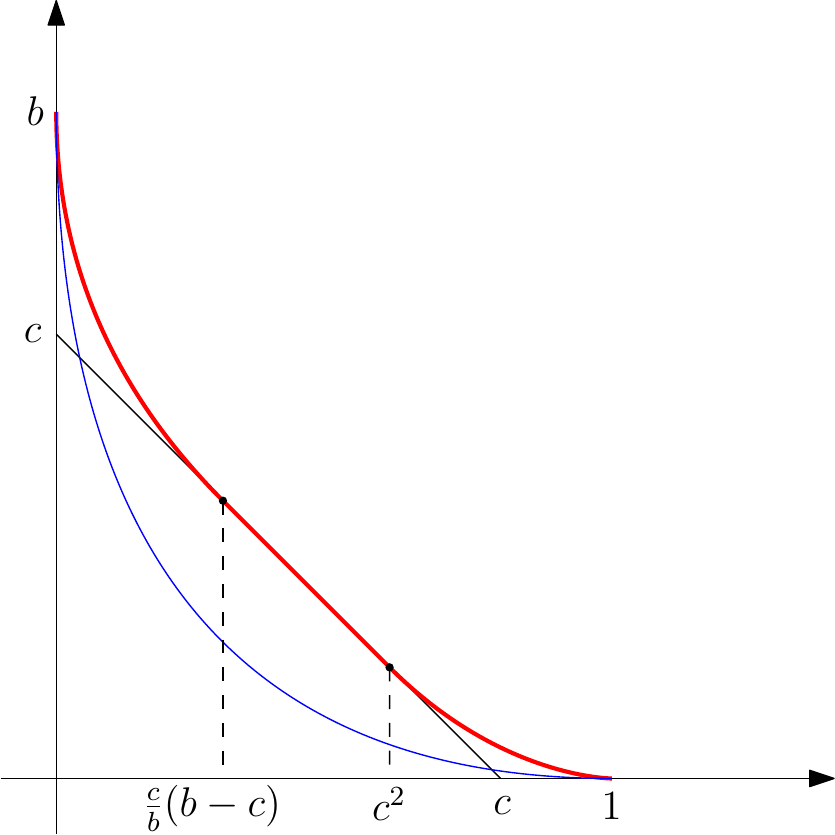}
\centering
\caption{Graph of $f_c$ in red.} \label{figure-4}
\end{figure}

Denote by $X_{f_c}$ the monotone toric domain such that $f_c$ is the boundary of its moment map minus the $w_1$ and $w_2$-axis. Then we have
\[\frac{c_{\rm Gr}(X_{\Omega})}{\text{Vol}(X_{\Omega})}\leq\frac{c_{\rm Gr}(X_{f_c})}{\text{Vol}(X_{f_c})}.\]
Meanwhile, by integrating along the graph $f_c(w_1)$, we get that 
\[\text{Vol}(X_{f_c}) = \frac{c^2}{2}+\frac{(b-c)^2c}{6b}+\frac{c(1-c)^2}{6}, \]
therefore, 
\[\frac{c_{\rm Gr}(X_{f_c})}{\text{Vol}(X_{f_c})}=\frac{c}{\text{Vol}(X_{f_c})}=\frac{6}{3c+\frac{(b-c)^2}{b}+(1-c)^2}.\]

Moreover, since $X_\Omega$ is geometrically convex with the $w_1$-intercept and $w_2$-intercept being $a(\Omega)=1$ and $b(\Omega)=b$, respectively, we have $\frac{b}{1+b}\leq c_{\rm Gr}(X_\Omega)\leq 1$, since the Gromov width of the domain with boundary $w_2=b(1-\sqrt{w_1})^2$ is $\frac{b}{1+b}$. Therefore, we get
\[\max\limits_{c\in\left[\frac{b}{1+b},1\right]}\left\{\frac{c_{\rm Gr}(X_{f_c})}{\text{Vol}(X_{f_c})}\right\}=\frac{6}{1+b}\]
where the maximum is obtained for $c=\frac{b}{1+b}$ i.e. for the domain whose boundary is given by $f(w_1)=b(1-\sqrt{w_1})^2$. Hence,
\[\text{ru}(X_\Omega)^2 \cdot \text{sys}(X_\Omega)=\frac{(1+b)^2c_{\rm Gr}(X_\Omega)^2}{4\text{Vol}(X_\Omega)^2}\leq 9.\]
Thus we completed the proof of the second conclusion. 
\end{proof}

\vspace*{3mm}
\bibliographystyle{amsplain}
\bibliography{biblio_convex}
\vspace*{3mm}

\end{document}